\documentclass[12pt]{amsart}
\title{On a Conjecture for $\aleph_0$-bounded Groups} 
\author{Marion Scheepers}
\date{}

\usepackage{amsfonts}
\usepackage{amsmath, amsthm}
\usepackage{colonequals}
\usepackage{txfonts}
\usepackage{color}
\newtheorem{theorem}{{\bf Theorem}}
\newtheorem{proposition}[theorem]{{\bf Proposition}}
\newtheorem{lemma}[theorem]{{\bf Lemma}}
\newtheorem{corollary}[theorem]{{\bf Corollary}}
\newtheorem{conjecture}{{\bf Conjecture}}

\newtheorem{problem}{{\bf Problem}}

\newtheorem{GSproblem}{{\bf Problem}}

\newcommand{\naturals}{{\mathbb N}}

\newcommand{\reals}{{\mathbb R}}

\newcommand{\sone}{{\sf S}_1}

\makeatletter

\begin{document}
\maketitle
\begin{abstract}
We show that it is consistent, relative to the consistency of a strongly inaccessible cardinal, that an instance of the generalized Borel Conjecture introduced in \cite{GS} holds while the classical Borel Conjecture fails.
 \end{abstract}


In \cite{Borel} E. Borel conjectured that when a subset $X$ of the additive group of real numbers, $(\reals,+)$ has the property that for each sequence $(I_n:n\in\naturals)$ of open neighborhoods of the identity element $0$ there is a sequence $(x_n:n\in\naturals)$ of real numbers such that $X\subseteq \bigcup_{n\in\naturals}x_n+I_n$, then $X$ is countable. This conjecture about the real line and its subsets has been considered from several points of view, leading to several characterizations of the subsets of the real line that satisfy Borel's hypotheses. In the early 20th century these sets were said to be sets with property $\textsf{C}$, but by the 1970's the terminology has changed to sets with \emph{strong measure zero}. Generalizing to mathematical structures beyond  the real line, and beyond the context of metrizable spaces, led to yet another change in terminology:
A subset $X$ of a topological group $(G,\odot)$ is said to be \emph{Rothberger bounded} if there is for each sequence $(I_n:n\in\naturals)$ of neighborhoods of the identity element $\textsf{id}_G$ of $G$, a sequence $(x_n:n\in\naturals)$ of elements of $G$ such that $X\subseteq \bigcup_{n\in\naturals} x_n\odot I_n$. 

Rothberger boundedness is an example of the following selection principle, $\textsf{S}_1(\mathcal{A},\mathcal{B})$, where $\mathcal{A}$ and $\mathcal{B}$ are families of sets:
\begin{quote}
For each sequence $(A_n:n\in\naturals)$ of elements of $\mathcal{A}$, there is a sequence $(b_n:n\in\naturals)$ such that: For each $n$, $b_n$ is an element of $A_n$, and the set $\{b_n:n\in\naturals\}$ is an element of $\mathcal{B}$.
\end{quote}
To see that Rothberger boundedness of a subset $X$ of a topological group $(G,\odot)$ is an instance of this selection principle, define the following two families $\mathcal{O}_{nbd}$ and $\mathcal{O}_X$: For 
an open neighborhood $N$ of the identity element $\textsf{id}_G$, let $x\odot N$ denote the set $\{x\odot a:\; a\in N\}$, and let $\mathcal{O}(N)$ denote the set $\{x\odot N:x\in G\}$, an open cover of $G$. Then $\mathcal{O}_{nbd}$ denotes the set $\{\mathcal{O}(U): U \mbox{ an open neighborhood of }\textsf{id}_G\}$ of all open covers of $G$ obtainable in this way. Second, for a subset $X$ of the group $G$, the symbol $\mathcal{O}_X$ denotes the collection whose elements are covers of $X$ by sets open in $G$. Then $\sone(\mathcal{O}_{nbd},\mathcal{O}_X)$ states that $X$ is a Rothberger bounded subset of the group $(G,\odot)$.

Thus, Borel's Conjecture is the statement that for the group of real numbers with the operation of addition, this selection principle holds for a subset $X$ if, and only if, $X$ is countable. Studies of Borel's Conjecture has led to the following reformulation:

\begin{theorem}\label{BCgroups}
Borel's Conjecture is equivalent to the statement that for each second countable $\textsf{T}_0$ topological group, each Rothberger bounded subset is countable.
\end{theorem}

Some of the mathematical facts behind this reformulation of the Borel Conjecture include that $\textsf{T}_0$ topological groups are at least $\textsf{T}_{3\frac{1}{2}}$, and that a first countable $\textsf{T}_2$ topological group is necessarily metrizable: These two facts are results of Birkhoff \cite{Birkhoff} and of Kakutani \cite{Kakutani}, and a 
contemporary presentation can be found in Theorems II.8.2 and II.8.3 of \cite{HRI}. From this point on we assume without further mention that all topological groups discussed in this paper are at least $\textsf{T}_0$. 

Towards generalizing the Borel Conjecture to a wider class of topological groups, call a topological group $(G,\odot)$ $\aleph_0$-\emph{bounded} if there is for each neighborhood $N$ of the identity element of $G$ a countable sequence $(x_n:n\in\naturals)$ of $G$ such that $G = \bigcup_{n\in\naturals}x_n\odot N$.
The notion of an $\aleph_0$-bounded group is due to Guran \cite{Guran}, who proved the following fundamental fact: 
\begin{theorem}[Guran]\label{guranembth} A topological group is $\aleph_0$-bounded if, and only if, it embeds as a topological group into a product of second countable topological groups.
\end{theorem}
Note that second countable topological groups are separable and metrizable. The topology on a product of topological spaces is taken to be the Tychonoff product topology.
The class of $\aleph_0$-bounded groups has nice preservation properties: Every subgroup of an $\aleph_0$-bounded group is $\aleph_0$-bounded, any (finite or infinite) Tychonoff product of $\aleph_0$-bounded groups is $\aleph_0$-bounded, every continuous homomorphic image of an $\aleph_0$-bounded group is $\aleph_0$-bounded, and if a dense subgroup of a group is $\aleph_0$-bounded, then so is the group. The survey \cite{MT} gives a good introduction to $\aleph_0$-bounded groups, and also contains a proof of the following quantified form of Guran's Theorem:

\begin{theorem}\label{guranquantify} For an $\aleph_0$-bounded topological group $(G,\odot)$ and an infinite cardinal number $\kappa$ the following are equivalent:
\begin{enumerate}
  \item{The weight of $G$ is $\kappa$.}
  \item{The character of $G$ is $\kappa$.} 
  \item{$\kappa$ is the smallest infinite cardinal such that $G$ embeds as a topological group into a product of $\kappa$ separable metrizable topological groups.}
\end{enumerate}
\end{theorem}

In this paper we report on the following generalization of Borel's Conjecture:
\begin{conjecture}\label{GBC}
In any $\aleph_0$-bounded group, the cardinality of a Rothberger bounded subset is no larger than the weight of the group.
\end{conjecture}
Note that Conjecture \ref{GBC} implies the Borel Conjecture. The Borel Conjecture has been proven independent of the standard axioms of Mathematics, namely the Zermelo-Fraenkel axioms, including the Axiom of Choice. The symbol \textsf{ZFC} denotes this axiom system, and we shall assume the consistency of \textsf{ZFC} in this paper. Thus, Conjecture \ref{GBC} is not a theorem of \textsf{ZFC}: As the negation of Borel's Conjecture is consistent relative to the consistency of \textsf{ZFC}, so is the negation of Conjecture \ref{GBC}. At this point it is not known whether Conjecture \ref{GBC} is also consistent, relative to the consistency of \textsf{ZFC}.

To further discuss what is currently known, and to frame our upcoming results, we introduce the following two notions: For infinite cardinal $\kappa$ let $\textsf{BC}_{\kappa}$ denote the following instance of Conjecture \ref{GBC}\\
\vspace{0.1in}

\begin{tabular}{l} 
 {\tt Each Rothberger bounded subset of an $\aleph_0$-bounded group } \\
                      {\tt of  weight $\kappa$ has cardinality at most $\kappa$}.
\end{tabular} 
\vspace{0.1in}

Define the class $\textsf{B}$ of cardinals as follows:
\[
  \textsf{B} = \{\kappa: \kappa \mbox{ is an infinite cardinal number and } \textsf{BC}_{\kappa} \mbox{ holds}\}.
\]
Conjecture \ref{GBC} states that $\textsf{B}$ is the class of all cardinals.
Little is known about the class $\textsf{B}$ of cardinals. In Section \ref{sect:prior}  we briefly survey results on $\textsf{B}$ obtained in \cite{GS}. This section is followed by an exposition of some new findings regarding $\textsf{B}$. 

\section{Prior Results}\label{sect:prior}

The symbol $\textsf{BC}_{\aleph_0}$ denotes Borel's Conjecture. 
The failure of the single instance $\textsf{BC}_{\aleph_0}$ implies the absence of $\aleph_0$ from the set $\textsf{B}$, and thus the failure of Conjecture \ref{GBC}. Sierpi\'nski proved in \cite{sierpinski} that the Continuum Hypothesis, abbreviated \textsf{CH}, implies the failure of  the instance $\textsf{BC}_{\aleph_0}$ of Conjecture \ref{GBC}. One might wonder just how badly Conjecture \ref{GBC} could fail. It was shown in \cite{GS} that it is consistent that each instance of Conjecture \ref{GBC} fails - i.e., $\textsf{B}=\emptyset$:
\begin{theorem}\label{failurecon} It is consistent, relative to the consistency of {\sf ZFC}, that ${\sf BC}_{\kappa}$ fails for each infinite cardinal number $\kappa$.
\end{theorem} 

In \cite{laver} R. Laver proved that it is consistent that the instance $\textsf{BC}_{\aleph_0}$ of Conjecture \ref{GBC} holds - i.e., $\aleph_0\in\textsf{B}$. Towards determining if any additional instances of Conjecture \ref{GBC} might hold, recall that a cardinal number $\kappa$ is said to be 1-\emph{inaccessible} if it is inaccessible, and there are $\kappa$ many inaccessible cardinal numbers less than $\kappa$. The following consistency result was obtained in \cite{GS}:
\begin{theorem}\label{conbelowomega2} If it is consistent that there is a  1-inaccessible cardinal, then {\sf ZFC}  plus $\textsf{BC}_{\aleph_0}+\textsf{BC}_{\aleph_1}+ 2^{\aleph_1}=\aleph_2$  is consistent. 
\end{theorem}
In particular, the higher cardinal versions of \textsf{CH} do not directly contradict the corresponding instances of Conjecture \ref{GBC}. Although it is not known if in Theorem \ref{conbelowomega2} the hypothesis of consistency of the existence of a 1-inaccessible cardinal is necessary, it is known that the consistency of the existence of an inaccessible cardinal is necessary. Thus, consistency of the statement $\{\aleph_0,\;\aleph_1\}\subseteq\textsf{B}$ implies the consistency of the existence of an inaccessible cardinal.

In \cite{GS} it was also shown that with a modest increase in the strength of consistency hypotheses, the consistency of the first $\omega$ instances of Conjecture \ref{GBC} is achievable:
\begin{theorem}\label{fragment2}
If it is consistent that there is a $1$-inaccessible cardinal with countably many inaccessible cardinals above it, then $\textsf{ZFC} + (\forall n<\omega){\sf BC}_{\aleph_n} + \neg\textsf{BC}_{\aleph_{\omega}}$ is consistent. 
\end{theorem}
In other words: If it is consistent that there is  a $1$-inaccessible cardinal with countably many inaccessible cardinals above it then it is consistent that $\{\aleph_n:n<\omega\}\subseteq \textsf{B}$ and $\aleph_{\omega}\not\in \textsf{B}$. It is currently not known if the consistency of $\{\aleph_n:n<\omega\}\subseteq \textsf{B}$ implies the consistency of the existence of a $1$-inaccessible cardinal with countably many inaccessible cardinals above it.

The first significant obstacle to obtaining the consistency of an instance of Conjecture \ref{GBC} appeared at the cardinal $\aleph_{\omega}$. In \cite{GS} the following consistency result is obtained.
\begin{theorem}\label{alephomandBC} If it is consistent that there is a 2-huge cardinal, then it is consistent that ${\sf BC}_{\aleph_0}$ as well as ${\sf BC}_{\aleph_1}$, and ${\sf BC}_{\aleph_{\omega}}$.
\end{theorem}
In other words: If it is consistent that there is a 2-huge cardinal, then it is consistent that $\{\aleph_0,\; \aleph_1,\;\aleph_{\omega}\}\subseteq \textsf{B}$. It is not known, but it seems unlikely, that these three instances of Conjecture \ref{GBC} holding simultaneously has the consistency strength of the existence of a 2-huge cardinal. However, it is known that the instances $\textsf{BC}_{\aleph_0}$ and $\textsf{BC}_{\aleph_{\omega}}$ of Conjecture \ref{GBC} holding simultaneously has significant consistency strength: It was pointed out in \cite{GS} that if $2^{\aleph_0}<\aleph_{\omega}$ and $\textsf{BC}_{\aleph_{\omega}}$ holds, then the Axiom of Projective Determinacy holds. We note that this statement is not the optimal that can be given with current knowledge, but is merely offered as an illustration.

The following result, proven in \cite{GS}, obtains the consistency of instances of Conjecture \ref{GBC} simultaneously at a proper class of cardinals.
\begin{theorem}\label{properclassBC}
If it is consistent that there is a 3-huge cardinal, then it is consistent that ${\sf BC}_{\aleph_0}$ as well as ${\sf BC}_{\aleph_1}$, and there is a proper class of cardinals $\kappa$ such that $\omega={\sf cf}(\kappa)$, and ${\sf BC}_{\kappa}$ as well as ${\sf BC}_{\kappa^+}$. 
\end{theorem}
In other words, if it is consistent that there is a $3$-huge cardinal, then it is consistent that there is a proper class of cardinals $\kappa$ of countable cofinality such that $\kappa,\; \kappa^+\in\textsf{B}$.  
To our knowledge, this is the current status of Conjecture \ref{GBC}.  
In the remaining parts of the paper we report results about some of the open problems raised in \cite{GS}. Some of these results rely on an equivalent form of Conjecture \ref{GBC} obtained in Theorem 11 of \cite{GS}. 

Towards stating this result we recall two concepts: 
For an infinite cardinal number $\kappa$ a family $\mathcal{F}$ of countable subsets of $\kappa$ is said to be a $(\kappa,\;\aleph_0)$ Kurepa family if $\vert\mathcal{F}\vert >\kappa$ and for each countable subset $A$ of $\kappa$ the set $\{X\cap A:X\in\mathcal{F}\}$ is countable. The symbol $\textsf{KH}(\kappa,\aleph_0)$ denotes the statement that there is a $(\kappa,\aleph_0)$ Kurepa family.

Aside from considering an instance of Conjecture \ref{GBC} for a specific infinite cardinal number, we also consider instances of Conjecture \ref{GBC} for specific groups. In this vein, the notation $\textsf{BC}(G,\odot)$ denotes the statement that each Rothberger bounded subset of the topological group $(G,\odot)$ has cardinality no larger than the weight of this group.

\begin{theorem}[\cite{GS}, Theorem 11]\label{KHequivalence}
The following statements are equivalent:
\begin{enumerate}
\item{Conjecture \ref{GBC}}
\item{$\textsf{BC}_{\aleph_0}$ $+$ $(\forall \kappa>\aleph_0)(\textsf{BC}(^{\kappa}2,\oplus))$}
\item{$\textsf{BC}_{\aleph_0}$ $+$ $(\forall \kappa>\aleph_0)(\neg\textsf{KH}(\kappa,\aleph_0))$}
\end{enumerate}
\end{theorem}

It is also important for one of our upcoming results that in the absence of $\textsf{BC}_{\aleph_0}$, for an uncountable cardinal $\kappa$ the following implications hold:
\begin{proposition}[\cite{GS}, Theorem 11]\label{implications} For each uncountable cardinal number $\kappa$, each of the following statements implies the next one:
\begin{enumerate}
\item{$\textsf{BC}_{\kappa}$}
\item{$\textsf{BC}(^{\kappa}2,\oplus)$}
\item{$\neg\textsf{KH}(\kappa,\aleph_0))$}
\end{enumerate}
\end{proposition}
It is also noted in \cite{GS}, Theorem 11, that in the presence of $\textsf{BC}_{\aleph_0}$ the three statements in Proposition \ref{implications} are equivalent. 

\section{$\textsf{BC}_{\aleph_0}$ is not a necessary condition for other instances of Conjecture \ref{GBC}}

In all examples in \cite{GS} of the consistency of instances of Conjecture \ref{GBC} it is also the case that $\textsf{BC}_{\aleph_0}$ holds. This state of affairs begs the question whether $\textsf{BC}_{\aleph_0}$ is necessary for any instance $\textsf{BC}_{\kappa}$ for some uncountable cardinal $\kappa$. This question appears as Problem 2 in \cite{GS}:  
{\flushleft{} 
\emph{Is it consistent that ${\sf BC}_{\kappa}$ holds for some uncountable cardinal $\kappa$, while ${\sf BC}_{\aleph_0}$ fails? What if $\kappa = \aleph_1$ or $\kappa = \aleph_{\omega}$}?}
\vspace{0.1in}

\begin{theorem}\label{GSProblem2}
If it is consistent that there is an inaccessbile cardinal then it is consistent  that $\neg\textsf{BC}_{\aleph_0} + \textsf{BC}_{\aleph_3}+\neg\textsf{BC}_{\kappa}$ for regular uncountable cardinals $\kappa\neq \aleph_3$.
\end{theorem}

For convenience, before proving Theorem \ref{GSProblem2} , we present in three parts the basic facts exploited in the proof. For 
a topological group $(G,\odot)$, 
define $\textsf{RB}(G,\odot)$ to be the least cardinal number $\kappa$ such that every Rothberger bounded subset of the group $(G,\odot)$ has cardinality at most $\kappa$.
In the earlier notation, for each group $(G,\odot)$, the statement $\textsf{BC}(G,\odot)$ is equivalent to the statement that $\textsf{RB}(G,\odot) \le weight(G,\odot)$. The Borel Conjecture is equivalent to the statement that $\textsf{RB}({\mathbb{R}},+) = \aleph_0$. It is evident that $\textsf{RB}({\mathbb{R}},+)$ is no larger than the continuum.

\begin{center}{\underline{Part 1: Bounding $\textsf{RB}(G,\odot)$ for separable metrizable groups.}} \end{center}
\vspace{0.1in}

Recall that a function $f$ from a metric space $(X,d)$ to a metric space $(Y,\rho)$ is a Lipschitz function if there is a positive real number $C$ such that for all $x$ and $y$ in $X$ we have $\rho(f(x),f(y))<C\cdot d(x,y)$. 
\begin{lemma}[Carlson \cite{CarlsonBC}]\label{carlson} If $(X,d)$ is a separable metric space of cardinality less than $2^{\aleph_0}$, then there is a one-to-one Lipschitz function from $X$ to $\reals$, the real line.
\end{lemma}

 The following Lemma reformulates Theorem 3.2 of \cite{CarlsonBC} for our purposes.
\begin{lemma}\label{smzsepmetric} If $\textsf{RB}({\mathbb R},+)^+ < 2^{\aleph_0}$, then in any separable metric space a strong measure zero set has cardinality at most $\textsf{RB}({\mathbb R},+).$
\end{lemma}
\begin{proof}
Let $(X,d)$ be a separable metric space and let $S$ be a subset of cardinality larger than $\textsf{RB}({\mathbb R},+)$, but less than $2^{\aleph_0}$. Then by Lemma \ref{carlson} fix a one-to-one Lipschitz function $f :S\rightarrow\reals$, and let $C>0$ be a constant witnessing the Lipschitz condition for $f$. Then the set $f\lbrack S\rbrack$ of real numbers has cardinality $\vert S\vert >  \textsf{RB}({\mathbb R},+)$.

 Suppose that contrary to the claim $S$ has strong measure zero.
Then $f\lbrack S\rbrack$ is a subset of $\reals$ of cardinality larger than $\textsf{RB}({\mathbb R},+)$, yet Rothberger bounded, a contradiction.
\end{proof}

\begin{corollary}\label{rbnumbertransfer} If $\textsf{RB}({\mathbb R},+)^+ < 2^{\aleph_0}$, then for any separable metrizable group $(G,\odot)$ we have $\textsf{RB}(G,\odot)\le \textsf{RB}({\mathbb R},+)$
\end{corollary}

\begin{center}{\underline{Part 2: Consistency of $\textsf{RB}(G,\odot) =\aleph_1< 2^{\aleph_0}$ for separable metrizable groups.}} \end{center}
\vspace{0.1in}

Though a number of consistency results regarding existence and possible values of $\textsf{RB}({\mathbb R}, +)$ are available, we mention only the following one, relevant to the current topic. 
\begin{lemma}[Bartoszynski-Judah \cite{BJ}, Theorem 2.15] \label{l:exclrand} After adding $\kappa>\aleph_1$ random reals to a model of CH,  $2^{\aleph_0} \ge \kappa>\aleph_1$ and $\textsf{RB}({\mathbb R},+) =\aleph_1$. 
\end{lemma}
Combining Corollary \ref{rbnumbertransfer} and Lemma \ref{l:exclrand}  
 we find
\begin{proposition}\label{exclrand} After adding $\kappa>\aleph_2$ random reals to a model of CH,  $2^{\aleph_0} \ge \kappa>\aleph_2$ and for every separable metrizable topological group $(G,\odot)$ it is true that  $\textsf{RB}(G,\odot) =\aleph_1$. 
\end{proposition}

\begin{center}{\underline{Part 3: Treating $\aleph_0$-bounded groups of uncountable weight.}} \end{center}
\vspace{0.1in}

Towards the next step, we first recall a generalization of $\textsf{KH}(\kappa,\aleph_0)$. For $\kappa>\lambda$ infinite cardinal numbers, a family $\mathcal{F}\subseteq \mathcal{P}(\kappa)$ is said to be a $(\kappa,\lambda)$ Kurepa family if $\vert\mathcal{F}\vert >\kappa$ while for each subset $S$ of $\kappa$ for which $\vert S\vert = \lambda$ we have $\vert\{X\cap S:X\in \mathcal{F}\}\vert\le \lambda$. The symbol $\textsf{KH}(\kappa,\lambda)$ denotes the statement that there exists a $(\kappa,\lambda)$ Kurepa family.
\begin{theorem}[Jensen]\label{JensenTh} If \textsf{V = L}, then $\textsf{KH}(\kappa,\lambda)$ holds for all infinite cardinals $\lambda < \kappa$ when $\kappa$ is a regular cardinal or a cardinal of countable cofinality.
\end{theorem}
A proof of this result may be found in \cite{Devlin}, Theorems VII.3.2 and VII.3.3. 

{\flushleft{\underline{\bf Step 3.1:}}} Starting with $\textsf{V}=\textsf{L}$, consider the generic extension $\textsf{L}\lbrack G\rbrack$ obtained by adding $\kappa$ random reals for some regular cardinal $\kappa>\aleph_2$.
In $\textsf{L}\lbrack G\rbrack$ we have that $2^{\aleph_0} = \kappa$ and by Lemma \ref{exclrand} that for each second countable group $(G,\odot)$ the equation $\textsf{RB}(G,\odot) = \aleph_1$ holds. 

Next we require the following fact about generic extensions, also known as the \emph{approximation lemma} - see \cite{Kunen}, Lemma IV.7.8:
\begin{lemma}\label{approximation} Let $\theta$ be an uncountable cardinal number.
Let $\mathbb{P}$ be a partially ordered set in which each pairwise incomparable set has cardinality less than $\theta$. 
Let $G$  be a $\mathbb{P}$-generic filter (over the ground model). If $A$ and $B$ are (ground model) sets and $f:A\longrightarrow B$ is a function in the generic extension, then there is a ground model function $F:A\longrightarrow\mathcal{P}(B)$ such that in the ground model for each $a\in A$ we have $\vert F(a)\vert <\theta$, and in the generic extension, for each $a\in A$ it is the case that $f(a)\in F(a)$.
\end{lemma}

With $\mathbb{P}$ being the partially ordered set for adding a number of random reals, it follows from Lemma \ref{approximation} that every set of ordinals in $\textsf{L}\lbrack G\rbrack$ of cardinality $\aleph_1$ is contained in a set of ordinals in $\textsf{L}$ of cardinality $\aleph_1$. Thus, in the generic extension $\textsf{L}\lbrack G\rbrack$ the statement $\textsf{KH}(\kappa,\aleph_1)$ is still true for each regular cardinal $\kappa>\aleph_1$. Similarly, $\textsf{KH}(\kappa,\aleph_0)$ is still true for each regular cardinal $\kappa\ge \aleph_1$. It follows from Theorem \ref{implications} that in $\textsf{L}\lbrack G\rbrack$ the instance $\textsf{BC}_{\kappa}$ fails for each infinite regular cardinal $\kappa$.

{\flushleft{\underline{\bf Step 3.2:}}} Starting with the model $\textsf{L}\lbrack G\rbrack$ from Step 3.1, letting $\mu$ be an inaccessible cardinal, force next with the Levy Collapse $\textsf{Lv}(\mu,\aleph_3)$. A good overview of the Levy collapse is provided on pp. 126 - 131 of \cite{Kanamori}. In the resulting model we have:
\begin{enumerate}
\item{$\mu = \aleph_3$}
\item{$\textsf{KH}(\aleph_3,\aleph_0)$ as well as $\textsf{KH}(\aleph_3,\aleph_1)$ fail.}
\item{$2^{\aleph_0} = 2^{\aleph_1} = 2^{\aleph_2} = \aleph_3$}
\item{$\textsf{RB}({\mathbb R},+) =\aleph_1$ and }
\item{For regular uncountable $\mu\neq\aleph_3$, ${\textsf{KH}}(\mu,\aleph_0)$ as well as ${\textsf{KH}}(\mu,\aleph_1)$ hold.}
\end{enumerate}

It follows that in this generic extension we have $\neg\textsf{BC}_{\kappa}$ for each regular cardinal $\kappa\neq \aleph_3$.

{\flushleft{\underline{\bf Step 3.3:}}} Next we show that in this generic extension $\textsf{BC}_{\aleph_3}$ holds:\\

For an infinite cardinal $\kappa$ and for a subset $C$ of $\kappa$, if $S$ is a subset of $\prod_{\alpha<\kappa}X_{\alpha}$, then $S\lceil_C$ denotes the set $\{f\lceil_C:f\in S\}$. 
 \begin{lemma} [\cite{GS}, Lemma 4]\label{GS4} Let $\kappa$ be an infinite cardinal number, and let $(G_{\alpha}:\alpha<\kappa)$ be a family of topological groups. Let $X$ be a subset of the Tychonoff product $\Pi_{\alpha<\kappa}G_{\alpha}$. The following are equivalent:
 \begin{enumerate}
 \item{$X$ is a Rothberger bounded subset of $\Pi_{\alpha<\kappa}G_{\alpha}$.}
 \item{For each countable subset $C$ of $\kappa$, $X\lceil_C$ is a Rothberger bounded subset of $\Pi_{\alpha\in C}X_{\alpha}$.}
 \end{enumerate}
 \end{lemma}

Let $(G,\odot)$ be an $\aleph_0$-bounded topological group of weight $\aleph_3$, and let $X$ be a Rothberger bounded subset of $G$. By Theorem \ref{guranquantify} we find $\aleph_3$ separable metrizable topological groups $G_{\alpha},\; \alpha<\aleph_3$ such that $G$ embeds as subgroup into the product $\Pi_{\alpha<\aleph_3}G_{\alpha}$. Under this image $X$ is a Rothberger bounded subset of this product, and thus for each countable subset $C$ of $\aleph_3$, $X\lceil_C$ is a Rothberger bounded subset of $\Pi_{\alpha\in C}G_{\alpha}$, which as a product of countably many separable metrizable spaces is separable and metrizable, and thus by Corollary \ref{rbnumbertransfer}, $\vert X\lceil_C\vert\leq \aleph_1$. But then $\vert X\vert \le (\aleph_3^{\aleph_0})^{\aleph_1} = \aleph_3$. This completes the proof of Theorem \ref{GSProblem2}.

\section{In the presence of $\textsf{BC}_{\aleph_0}$, $\textsf{BC}_{\kappa}$ may hold for only one uncountable regular cardinal number.}

\begin{theorem}\label{LIntervals}
Assume the consistency of $\textsf{ZFC}+$ there is an inaccessible cardinal. For each $n\in\naturals$ with $n>1$ it is consistent that $\textsf{BC}_{\aleph_0} + \textsf{BC}_{\aleph_n}$ while for any other uncountable regular cardinal $\kappa$, $\neg\textsf{BC}_{\kappa}$. 
\end{theorem}

\begin{proof}
We organize the proof in two steps.
{\flushleft{\bf Step 1: }} The generic extension $\textsf{L}\lbrack G\rbrack$ in which the only regular cardinal $\kappa$ for which the instance $\textsf{BC}_{\kappa}$ holds, is $\kappa=\aleph_0$.

Let ${\mathbb M}$ be the $\aleph_2$-step countable support iteration of the Mathias reals partially ordered set. By the results in Section 9 of \cite{baumgartner}, ${\mathbb M}$ is a proper partially ordered set. Moreover, if CH holds, then by Theorem 7.2 of \cite{baumgartner} ${\mathbb M}$ preserves $\aleph_1$ and has the $\aleph_2$-chain condition, and thus by Theorem IV.7.9 of \cite{Kunen}, ${\mathbb M}$ preserves cardinals.

Start with ground model $\textsf{V} = \textsf{L}$. The $\aleph_2$-step countable support iteration $\mathbb{M}$ of the Mathias reals poset results in the generic extension $\textsf{L}\lbrack G\rbrack$ in which Borel's Conjecture holds.  Lemma \ref{lm:pfa} is given as Lemma 31.4 in \cite{Jech}.
\begin{lemma}\label{lm:pfa}
If ${\mathbb P}$ is a proper partially ordered set and $G$ is ${\mathbb P}$ generic, then each countable set in $\textsf{V}\lbrack G\rbrack$ is a subset of a set in $\textsf{V}$ that is countable in $\textsf{V}$. 
\end{lemma}
Since ${\mathbb M}$ is a cardinal preserving proper partially ordered set, for all cardinal numbers $\kappa$ for which $\textsf{KH}(\kappa,\aleph_0)$ was true in $\textsf{L}$, we still have $\textsf{KH}(\kappa,\aleph_0)$ true in $\textsf{L}\lbrack G\rbrack$. By Proposition \ref{implications} and Theorem \ref{JensenTh}, for each regular uncountable cardinal $\kappa$ the instance $\textsf{BC}_{\kappa}$ is false in $\textsf{L}\lbrack G\rbrack$.

At this stage we observe that the cardinal arithmetic in $\textsf{L}\lbrack G\rbrack$ deviates from that in $\textsf{L}$ only in that $2^{\aleph_0} = 2^{\aleph_1} = \aleph_2$: For all $\kappa>\aleph_0$ we have $2^{\kappa}=\kappa^+$ in $\textsf{L}\lbrack G\rbrack$.

{\flushleft{\bf Step 2: }} For each integer $n>1$ there is a generic extension $\textsf{L}\lbrack G\rbrack\lbrack K\rbrack$ of $\textsf{L}\lbrack G\rbrack$  in which the only regular cardinals $\kappa$ for which an instance $\textsf{BC}_{\kappa}$ holds, are $\kappa=\aleph_0$ and $\kappa=\aleph_n$.

Fix an integer $n>1$. We now proceed as in Case 2 of Section 2 of \cite{F-G}, the only difference being that in \cite{F-G} the ground model is $\textsf{L}$ while in our case the ground model is $\textsf{L}\lbrack G\rbrack$. We leave it to the reader to check that Lemmas 2.1, 2.2, 2.3 and 2.7 of \cite{F-G} also hold over the ground model $\textsf{L}\lbrack G\rbrack$. 
It follows that in $\textsf{L}\lbrack G\rbrack\lbrack K\rbrack$ for regular uncountable cardinals $\mu\neq \aleph_n$ the statement $\textsf{KH}(\mu,\aleph_0)$ holds and $\textsf{KH}(\aleph_n,\aleph_0)$ fails.

Moreover by Lemma 2.2 of \cite{F-G} the partially ordered set used in the forcing extension over $\textsf{L}\lbrack G\rbrack$ preserves the cardinal equalities $2^{\aleph_0} = 2^{\aleph_1} = \aleph_2$ and does not add any new subsets of the real line, ${\mathbb R}$. In the generic extension $\textsf{L}\lbrack G\rbrack\lbrack K\rbrack$ the instance $\textsf{BC}_{\aleph_0}$ of Conjecture \ref{GBC} still holds.

Finally, applying Proposition \ref{implications} we find that in the generic extension $\textsf{L}\lbrack G\rbrack\lbrack K\rbrack$ the instances $\textsf{BC}_{\aleph_0}$ and $\textsf{BC}_{\aleph_n}$ of Conjecture \ref{GBC} hold, while for all other regular cardinals $\kappa$ the instances $\textsf{BC}_{\kappa}$ fail.  
\end{proof}
Note, incidentally, that in the generic extension in the proof of Theorem \ref{LIntervals}, we have for all cardinals $\kappa>\aleph_0$ that $2^{\kappa} = \kappa^+$.

\section{Conclusion}

There are numerous questions about instances of Conjecture 1 to which we do not know, at this time, answers. We mention only a few.

In \cite{GS}  Theorem \ref{failurecon} was proven by showing that $(\forall \kappa)(\neg {\sf  BC}(^{\kappa}2,\oplus))$ holds in generic extensions by $\aleph_1$ Cohen reals. Might any instances of Conjecture \ref{GBC} hold in the constructible universe? 
\begin{GSproblem}{\cite{GS}} Does ${\mathbf V}={\mathbf L}$ imply $(\forall \kappa)(\neg {\sf BC}(^{\kappa}2,\oplus))$?
\end{GSproblem}

It is expected that the answer to this problem is ``yes", but only fragments of this suspicion have been confirmed:
\begin{theorem}\label{LBC}
Assume that $\mathbf{ V} = \mathbf{L}$. For each cardinal $\kappa$ that is either regular, or singular of countable cofinality, $\neg\textsf{BC}_{\kappa}$.
\end{theorem}
\begin{proof}
 In Chapter VII.3 of \cite{Devlin} it is proven that for regular uncountable $\kappa$, $\textsf{KH}(\kappa,\aleph_0)$ holds in $\textsf{L}$. In Exercise VII.3 of \cite{Devlin}, it is also outlined how to prove that for an uncountable singular cardinal $\kappa$ of countable cofinality, the statement $\textsf{KH}(\kappa,\aleph_0)$ holds in $\textsf{L}$.
 By Theorem 11 of \cite{GS}, for an uncountable cardinal $\kappa$, $\textsf{KH}(\kappa,\aleph_0)$ implies the failure of $\textsf{BC}_{\kappa}$.
\end{proof}

Thus, to fully answer Problem 1, the following needs to be settled. \emph{ 
Assume $\mathbf{V}=\mathbf{L}$. If $\kappa$ is a singular cardinal of uncountable cofinality, does $\neg\textsf{BC}_{\kappa}$ hold?}

Theorem \ref{GSProblem2} partially answers Problem 2 of \cite{GS}. The techniques used in the proof of Theorem \ref{GSProblem2} would probably not adapt to for example deter- mine whether it is consistent (relative to the consistency of an appropriate cardinal hypothesis) that $\aleph_1$ is the least element of the set $\textsf{B}$ (i.e., $\textsf{BC}_{\aleph_0}$ fails while $\textsf{BC}_{\aleph_1}$ holds), or to determine if it is consistent (relative to the consistency of an appropriate cardinal hypothesis) that $\aleph_{\omega}$ is the least element of the set $\textsf{B}$ (i.e., $\textsf{BC}_{\aleph_n}$ fails for all $n<\omega$, while $\textsf{BC}_{\aleph_{\omega}}$ holds).

It was also pointed out that the consistency of the instance $\textsf{BC}(^{\aleph_{\omega}}2,\aleph_{\omega})$ was obtained from the consistency of the existence of a 2-huge cardinal. It would be of interest to know the exact consistency strength of $\textsf{BC}(^{\aleph_{\omega}}2,\aleph_{\omega})$.

We expect that Conjecture \ref{GBC} is consistent relative to the consistency of some large cardinal axioms. At present there is no indication of what large cardinal axiom might suffice. Perhaps an even more ambitious goal is:
\begin{problem} Determine the consistency strength of Conjecture \ref{GBC}.
\end{problem}

\section{Acknowledgments}

I would like to thank the organizers of the conference \emph{Frontiers in Se-lection Principles} that took place in Warsaw in late August 2017 for the opportunity to speak on generalized versions of Borel?s Conjecture. Scientific exchange at this conference provided impetus for writing this paper.

\end{document}